\documentclass[12pt]{amsart}
\usepackage{verbatim}
\usepackage{amscd}

\numberwithin{equation}{section}
\input epsf

\begin{document}

\renewcommand{\theequation}{\thesection.\arabic{equation}}
\setcounter{secnumdepth}{2}
\newtheorem{theorem}{Theorem}[section]
\newtheorem{definition}[theorem]{Definition}
\newtheorem{lemma}[theorem]{Lemma}
\newtheorem{corollary}[theorem]{Corollary}
\newtheorem{proposition}[theorem]{Proposition}
\numberwithin{equation}{section}
\theoremstyle{definition}
\newtheorem{example}[theorem]{Example}
\title[GIT aspect of GK Reduction. I]
{The GIT aspect of generalized K$\ddot{a}$hler reduction. I}

\author[Yicao Wang]{Yicao Wang}
\address
{Department of Mathematics, Hohai University, Nanjing 210098, China}
\thanks{ This study is supported by the Natural Science Foundation of Jiangsu Province, China (BK20150797) and by the China Scholarship Council (201806715027).}
\maketitle

\baselineskip= 20pt
\begin{abstract}We revisit generalized K$\ddot{a}$hler reduction introduced by Lin and Tolman in \cite{LT} from a viewpoint of geometric invariant theory. It is shown that in the strong Hamiltonian case introduced in the present paper, many well-known conclusions of ordinary K$\ddot{a}$hler reduction can be generalized without much effort to the generalized setting. It is also shown how generalized holomorphic structures arise naturally from the reduction procedure.

\end{abstract}
\section{Introduction}
Generalized complex geometry, initiated by N. Hitchin \cite{Hit} and developed in depth by M. Gualtieri \cite{Gu00} \cite{Gu0} \cite{Gu1}, is a simultaneous generalization of complex and symplectic geometries. The development of this new geometry thus benefits greatly from the two well-established disciplines. For example, Marsden-Weinstein's famous symplectic reduction is an important construction in symplectic geometry and has been successfully generalized to the generalized complex setting by several authors \cite{BCG1} \cite{LT}; in particular, generalized complex reduction introduced in \cite{LT} is precisely the analogue of symplectic reduction.

Symplectic reduction, when applied to Hamiltonian equivariant K$\ddot{a}$hler manifolds, is usually referred to as K$\ddot{a}$hler reduction. Besides the ordinary symplectic content of this procedure, complex structures also play a remarkable role; in particular, in good cases there is the well-known theorem of Kempf-Ness type stating that the symplectic quotient coincides with the complex quotient in the sense of geometric invariant theory (GIT for short) developed by Mumford. The philosophy of K$\ddot{a}$hler reduction has even been successfully applied formally to infinite-dimensional cases and the Kempf-Ness theorem also has its counterpart--the Kobayashi-Hitchin correspondence.

In the generalized complex setting, the analogue of K$\ddot{a}$hler manifolds are generalized K$\ddot{a}$hler manifolds, each consisting of two commuting generalized complex structures $\mathbb{J}_1$ and $\mathbb{J}_2$. The relevant generalized K$\ddot{a}$hler reduction was investigated in \cite{BCG1} \cite{BCG2} and in \cite{LT} independently. Lin-Tolman's work \cite{LT} is much more in the original spirit of symplectic reduction. However, in \cite{LT} attention was mainly put on one single generalized complex structure $\mathbb{J}_2$ which plays a role as a symplectic structure does in K$\ddot{a}$hler reduction. The effect of the other generalized complex structure $\mathbb{J}_1$ is thus not very clear. In contrast with K$\ddot{a}$hler reduction, it is conjectured that $\mathbb{J}_1$ should act like a complex structure in K$\ddot{a}$hler reduction and there should be an analogue of GIT quotient. Our goal in this paper is basicly to investigate this GIT aspect of generalized K$\ddot{a}$hler reduction.

However, to carry this idea out, we should complexify the underlying action of a compact Lie group $G$ first. There do exist two ordinary complex structures $J_\pm$ on the generalized K$\ddot{a}$hler manifold stemming from the bihermitian description, but they are not suitable for this attempt. Motivated by our central Lemma~\ref{lemma}, we use the generalized complex structure $\mathbb{J}_1$ to complexify the $G$-action. However, this procedure doesn't apply to general Hamiltonian generalized K$\ddot{a}$hler manifolds. To circumvent this difficulty we introduce the notion of \emph{strong Hamiltonian action of a compact Lie group}. For the strong Hamiltonian case, at least in the case that $M$ is compact and $G$ acts locally freely on the zero locus of the moment map, we can really prove that the main results of GIT do hold in this more general setting.

The paper is organized as follows: In $\S$~\ref{Kah}, we recall the basics of K$\ddot{a}$hler reduction to motivate our later investigation. In $\S$~\ref{gg} we collect the most relevant material on generalized geometry. In particular, a very brief review of Courant reduction of \cite{BCG1} is included. Our study really starts from $\S$~\ref{Hami}, in which we emphasize that our Lemma~\ref{lemma} is actually the generalized K$\ddot{a}$hler analogue of the basic formula Eq.~(\ref{grad}) in equivariant K$\ddot{a}$hler geometry. This motivates our attempt to complexify the group action at the infinitesimal level. We find the Poisson structure $\beta_1$ (which is zero in the K$\ddot{a}$hler case) associated to $\mathbb{J}_1$ is a basic obstruction for this complexification. In $\S$~\ref{stro}, we introduce the notion of strong Hamiltonian action of a compact Lie group. Several nontrivial examples are given. After that, in the compact case, we prove that the infinitesimal action can be integrated to a global one (Thm.~\ref{inte}). The goal of $\S$~\ref{GIT} is basically to establish a theorem of Kempf-Ness type in the strong Hamiltonian case, which claims that two kinds of generalized complex quotients are actually the same (Thm.~\ref{kempf} and Thm.~\ref{kempf2}). Our investigation follows closely the ideas of Kirwan in \cite{Kir}. In the same section, we also prove that a generalized holomorphic principal bundle arises naturally from the reduction procedure (Thm.~\ref{ghs}). This provides an answer to a question posed by the author in \cite{Wang2}, which suggests the possibility of constructing generalized holomorphic structures from generalized K$\ddot{a}$hler reduction. The last section $\S$~\ref{orbit} mainly contains a theorem (Thm.~\ref{Ob}) stating that a (stable) orbit of the complexified group $G^\mathbb{C}$ carries a natural structure of Hamiltonian $G$-K$\ddot{a}$hler manifold.

For simplicity, in the present paper we mainly concentrate on the special case that the group $G$ acts (locally) freely on the zero locus of the moment map. There is certainly a more general story, which we leave for a future work.
\section{Reflections on K$\ddot{a}$hler reduction}\label{Kah}
To motivate our later considerations, in this section we briefly review the ordinary K$\ddot{a}$hler reduction in equivariant K$\ddot{a}$hler geometry. We refer the interested readers to F. Kirwan's book \cite{Kir} for more details.

Let $G$ be a compact connected Lie group (with Lie algebra $\mathfrak{g}$) acting holomorphically on a K$\ddot{a}$hler manifold $(M, g, J)$ from the left in a Hamiltonian fashion. This means that $G$ preserves the K$\ddot{a}$hler form $\omega=gJ$ and there is an equivariant map $\mu: M\rightarrow \mathfrak{g^*}$ such that $d\mu_\varsigma=-\iota_{X_\varsigma}\omega$ where $\varsigma\in \mathfrak{g}$ and $X_\varsigma$ is the vector field on $M$ generated by $\varsigma$.

There is a unique complexification $G^\mathbb{C}$ of $G$ such that its Lie algebra is the complexification of $\mathfrak{g}$, i.e. $\mathfrak{g}_\mathbb{C}=\mathfrak{g}+\sqrt{-1}\mathfrak{g}$.\footnote{In this paper, $V_\mathbb{C}$ is used to denote the complexification of a real vector space or vector bundle $V$.} The $G$-action on $M$ can be extended to a $G^\mathbb{C}$-action under suitable conditions (e.g. $M$ is compact). To achieve this, at the infinitesimal level one simply uses $JX_\varsigma$ as the infinitesimal action generated by $\sqrt{-1}\varsigma$ for $\varsigma\in \mathfrak{g}$, which also preserves $J$.

If $0$ is a regular value of $\mu$ and $G$ acts freely on $\mu^{-1}(0)$, then the quotient $\mu^{-1}(0)/G$ acquires a symplectic structure by Marsden-Weinstein reduction. The complex structure $J$ also descends to this quotient and makes it a K$\ddot{a}$hler manifold. There is another way to view the reduced complex structure: $G^\mathbb{C}$ acts freely and holomorphically on the open set $M_s=G^\mathbb{C} \mu^{-1}(0)$, and thus the quotient $M_s/G^\mathbb{C}$ is complex in the natural manner. In this case, the Kempf-Ness theorem says that the two quotient actually coincide. \emph{To establish this coincidence, a central observation is that, while $X_\varsigma$ is the Hamiltonian vector field associated to $\mu_\varsigma$, $JX_\varsigma$ is minus the gradient vector field associated to $\mu_\varsigma$}, i.e.
\begin{equation}JX_\varsigma=-g^{-1}d\mu_\varsigma.\label{grad}\end{equation}

 Another fact, which is seldom mentioned in the literature, is that it is almost trivial that $M_s$, as a principal $G^\mathbb{C}$-bundle over $\mu^{-1}(0)/G$, carries a holomorphic structure. However, in the generalized setting, nontrivial generalized holomorphic structures are not easy to construct. So in our later investigation we will pay some attention to finding out whether generalized holomorphic structures arise in a similar manner.
\section{Basics of generalized complex geometry}\label{gg}
\subsection{Courant algebroids and their symmetries }\label{gr}We recall some backgrounds of Courant algebroids and extended actions of Lie groups on a Courant algebroid. Our basic reference is \cite{BCG1}.

Generalized geometry arises from the idea of replacing the tangent bundle $T$ of a manifold $M$ by the direct sum $\mathbb{T}$ of $T$ and its dual $T^*$, or, more generally, by an exact Courant algebroid.

A Courant algebroid $E$ is a real vector bundle $E$ over $M$, together with an anchor map $\pi$ to $T$,\footnote{Throughout the paper there are different Courant algebroids, but we always denote the anchor map by $\pi$. The context will exclude ambiguities.} a non-degenerate inner product $(\cdot, \cdot)$ and a so-called Courant bracket $[\cdot , \cdot]_c$ on $\Gamma(E)$. These structures should satisfy some compatibility axioms we won't mention here. $E$ is called exact, if the short sequence \[0\longrightarrow T^*\stackrel{\pi^*}\longrightarrow E \stackrel{\pi}\longrightarrow T \longrightarrow0\]
is exact. In this paper, by "Courant algebroid", we always mean an exact one. Given $E$, one can always find an isotropic right splitting $s:T\rightarrow E$, with a curvature form $H\in \Omega_{cl}^3(M)$ defined by
\[H(X,Y,Z)=([s(X),s(Y)]_c,s(Z)),\quad X, Y, Z\in \Gamma(T).\]
  By the bundle isomorphism $s+\pi^*:T\oplus T^*\rightarrow E$, the Courant algebroid structure can be transported onto $\mathbb{T}$. Then the inner product $(\cdot,\cdot)$ is the natural pairing, i.e.
$( X+\xi,Y+\eta)=\xi(Y)+\eta(X)$, and the Courant bracket is
\begin{equation}[X+\xi, Y+\eta]_H=[X,Y]+\mathcal{L}_X\eta-\iota_Yd\xi+\iota_Y\iota_XH,\label{Courant}\end{equation}
called the $H$-twisted Courant bracket. This bracket is not skew-symmetric, actually
\begin{equation}[X+\xi, Y+\eta]_H+[Y+\eta, X+\xi]_H=d(\xi(Y)+\eta(X)).\label{nonsk}\end{equation}
Different splittings are related by $B$-field transforms, i.e. $e^B(X+\xi)=X+\xi+\iota_XB$, where $B$ is a 2-form.

A Courant algebroid has more symmetries than the tangent bundle $T$. In a given splitting, an automorphism of $E$ is represented by a pair $(\psi, B)$, where $\psi$ is a diffeomorphism of $M$ and $B$ is a 2-form on $M$. These two should satisfy $H-\psi^*(H)=dB$.
The pair acts on $Y+\eta\in \Gamma(E)$ in the following manner: \[(\psi, B)\cdot (Y+\eta)=\psi_*^{-1}(Y)+\psi^*(\eta-\iota_YB).\] An infinitesimal automorphism is consequently a pair $(X, B)$, where $X$ is a vector field on $M$ and $B$ is a 2-form, satisfying $L_XH=-dB$. $(X, B)$ acts on $Y+\eta\in \Gamma(E)$ as follows: \[(X, B)\cdot (Y+\eta)=L_X(Y+\eta)-\iota_YB.\]
Especially, $X+\xi\in \Gamma(E)$ generates an infinitesimal inner automorphism $(X, d\xi-\iota_XH)$ through the Courant bracket (\ref{Courant}).

Let $G$ be a connected Lie group acting on $M$ from the left. Then the infinitesimal action $\varphi_0:\mathfrak{g}\rightarrow \Gamma(T)$ is a Lie algebra homomorphism. Since in generalized geometry $T$ is replaced by a Courant algebroid $E$, we would like to lift the $\mathfrak{g}$-action to $E$.
\begin{definition}A map $\varphi: \mathfrak{g}\rightarrow \Gamma(E)$ covering $\varphi_0$ is called an isotropic trivially extended $\mathfrak{g}$-action if (i) $\varphi$ is isotropic,i.e. the image of $\varphi$ is isotropic pointwise w.r.t. the inner product and (ii) $\varphi$ preserves the brackets, i.e. $\varphi([\varsigma,\zeta])=[\varphi(\varsigma), \varphi(\zeta)]_c$ for $\varsigma, \zeta\in \mathfrak{g}$.
\end{definition}
 Infinitesimal inner automorphisms of $E$ induced by $\varphi(\mathfrak{g})$ through the Courant bracket form a Lie algebra action of $\mathfrak{g}$ on the total space of $E$.\footnote{This still holds even without (i) because $[A,B]_c+[B,A]_c$ is an exact 1-form, which has no effect at the level of infinitesimal inner automorphism of $E$. The importance to include (i) in the definition will be clear later.} If this action integrates to a $G$-action, we shall call it an isotropic trivially extended $G$-action. There is a fairly general theory of extended $G$-action in \cite{BCG1}. However, we don't need this generality. \emph{In the remainder of the paper, when referring to an extended $\mathfrak{g}$-action ($G$-action), we always mean an isotropic trivially extended one.} Additionally, to simplify notation, $\varphi$ will be used to denote either an extended $\mathfrak{g}$-action or an extended $G$-action.

Given a splitting of $E$ preserved by the extended action, the extended $\mathfrak{g}$-action can be written in the form $\varphi(\varsigma)=X_\varsigma+\xi_\varsigma$, where $X_\varsigma$ is the vector field generated by $\varsigma$ and $\xi_{(,)}:\mathfrak{g}\rightarrow \Omega^1(M)$ is a $\mathfrak{g}$-equivariant map such that
\begin{equation}\xi_\varsigma(X_\zeta)+\xi_\zeta(X_\varsigma)=0,\quad d\xi_\varsigma=\iota_{X_\varsigma}H.\end{equation}

If the underlying $G$-action on $M$ of an extended $G$-action is proper and free, the Courant algebroid $E$ descends to the quotient $M/G$. In fact, let $K$ be the subbundle of $E$ generated by the image of $\varphi$ and $K^\bot\subset E$ the orthogonal of $K$ w.r.t. the inner product. Then $K\subset K^\bot$ and we can obtain a Courant algebroid $E_{red}:=\frac{K^\bot}{K}/G$ whose Courant bracket can be derived from the Courant bracket of $G$-invariant sections of $E$. If $\mathcal{D}\subset E_\mathbb{C}$ is an involutive isotropic subbundle, then $\mathcal{D}$ also descends to the quotient under good conditions, e.g. $\mathcal{D}\cap K_\mathbb{C}$ has constant rank. The reduced version of $\mathcal{D}$ is $\frac{\mathcal{D}\cap K_\mathbb{C}^\bot+K_\mathbb{C}}{K_\mathbb{C}}/G$. Of particular interest for us is the case of Dirac structures, involutive maximal isotropic subbundles of $E_\mathbb{C}$.
\subsection{Generalized holomorphic and generalized K$\ddot{a}$hler structures }
\begin{definition}
 A generalized complex structure on a Courant algebroid $E$ is a complex structure $\mathbb{J}$ on $E$ orthogonal w.r.t. the inner product and its $\sqrt{-1}$-eigenbundle $L\subset E_\mathbb{C}$ is involutive under the Courant bracket. We also say $\mathbb{J}$ is integrable in this case.
\end{definition}
Ordinary complex and symplectic structures are extreme examples of generalized complex structures. Note that the integrability of $\mathbb{J}$ is equivalent to
\begin{equation}[\mathbb{J} A, \mathbb{J} B]_c=\mathbb{J}[\mathbb{J} A, B]_c+\mathbb{J}[A, \mathbb{J} B]_c+[A, B]_c,\label{inteJ}\end{equation}
for any $A, B\in \Gamma(E)$. Since $\mathbb{J}$ and its $\sqrt{-1}$-eigenbundle $L$ are equivalent notions, we shall use them interchangeably to denote a generalized complex structure. A generalized complex structure $L$ is an example of complex Lie algebroids. Via the inner product, $\wedge^\cdot L^*$ can be identified with $\wedge^\cdot \bar{L}$, and we have an elliptic differential complex $(\Gamma(\wedge^\cdot \bar{L}), d_L)$, inducing the Lie algebroid cohomology associated with the Lie algebroid $L$. The differential complex can be further twisted by a generalized holomorphic vector bundle $F$, which is a complex vector bundle equipped with a first-order differential operator $\bar{\partial}: \Gamma(F)\rightarrow \Gamma(\bar{L}\otimes F)$ such that for any smooth function $f$ and $s\in \Gamma(F)$
\[\bar{\partial}(fs)=d_L f\otimes s+f\bar{\partial}s\]
 and $\bar{\partial}^2=0$. This notion of generalized holomorphic structures has been generalized to the setting of (generalized) principal bundles by the author in \cite{Wang1}.

\begin{definition}
Let $G$ be a Lie group. A generalized principal $G$-bundle over $M$ is a triple $(\mathbf{P}, \mathbf{E}, \mathbf{\varphi})$ such that (i) $p:\mathbf{P}\rightarrow M$ is an ordinary principal $G$-bundle, and (ii) $\mathbf{E}$ is a Courant algebroid over $\mathbf{P}$ and $\mathbf{\varphi}$ an extended $G$-action on $\mathbf{E}$.
\end{definition}
Note that by the reduction mentioned in \S~\ref{gr}, there is the Courant algebroid $E=\frac{K^\bot}{K}/G$ on the base manifold $M$, descending from $\mathbf{E}$. In the following, let $Q$ be a complex Lie group with Lie algebra $\mathfrak{q}$ and $(\mathbf{P}, \mathbf{E}, \mathbf{\varphi})$ a generalized principal $Q$-bundle. Then we have a decomposition $\mathfrak{q}_\mathbb{C}=\mathfrak{q}_h\oplus \mathfrak{q}_a$, where $\mathfrak{q}_h$ is the holomorphic part and $\mathfrak{q}_a$ its complex conjugate. Denote $K_a$ the sunbundle generated by $\varphi(\mathfrak{q}_a)$.
\begin{definition}\label{ghp}
Let $(\mathbf{P}, \mathbf{E}, \varphi)$ be a generalized principal $Q$-bundle over $M$ and $\mathcal{A}\subset \mathbf{E}_\mathbb{C}$ a $Q$-invariant isotropic subbundle such that (i) $K_a\subset \mathcal{A}\subset K^\bot_\mathbb{C}$, (ii) $\mathcal{A}\oplus\overline{\mathcal{A}}=K^\bot_\mathbb{C}$, and (iii) $\mathcal{A}$ descends to a generalized complex structure $L$ in $E$ in the sense of the reduction theory mentioned in \S~\ref{gr}. Then $\mathcal{A}$ is called an almost generalized holomorphic structure w.r.t. $L$. If furthermore $\mathcal{A}$ is integrable under the Courant bracket, it is called a generalized holomorphic structure.
\end{definition}
As was proved in \cite{Wang1}, given a generalized holomorphic principal $Q$-bundle, any associated vector bundle of a holomorphic representation of $Q$ is generalized holomorphic in the manner that associated vector bundles of a holomorphic principal $Q$-bundle are holomorphic.

\begin{definition}A generalized metric on a Courant algebroid $E$ is an orthogonal, self-adjoint operator $\mathcal{G}$ such that $( \mathcal{G}e,e) > 0$ for nonzero $e\in E$. It is necessary that $\mathcal{G}^2=id$.
 \end{definition}
 A generalized metric induces a canonical isotropic splitting: $E=\mathcal{G}(T^*)\oplus T^*$. It is called \emph{the metric splitting}. Given a generalized metric, we shall always choose the metric splitting to identify $E$ with $\mathbb{T}$. Then $\mathcal{G}$ is of the form $\left(\begin{array}{cc} 0 & g^{-1} \\g & 0 \\
\end{array} \right)$ where $g$ is an ordinary Riemannian metric.

  A generalized metric is an ingredient of a generalized K$\ddot{a}$hler structure, which is the analogue of K$\ddot{a}$hler structures in complex geometry.
\begin{definition}
A generalized K$\ddot{a}$hler structure on $E$ is a pair of commuting generalized complex structures $(\mathbb{J}_1,\mathbb{J}_2)$ such that $\mathcal{G}=-\mathbb{J}_1 \mathbb{J}_2$ is a generalized metric.
\end{definition}

If necessary, we will use $L_i$ to denote the corresponding $\sqrt{-1}$-eigenbundle of $\mathbb{J}_i$, $i=1,2$. A generalized K$\ddot{a}$hler structure can also be characterized in terms of more ordinary concepts: There are two complex structures $J_\pm$ on $M$ compatible with the metric $g$ induced from the generalized metric. Let $\omega_\pm=gJ_\pm$. Then in the metric splitting the generalized complex structures and the bihermitian data are related by the Gualtieri map:
\[\mathbb{J}_1=\frac{1}{2}\left(
  \begin{array}{cc}
    -J_+-J_-& \omega_+^{-1}-\omega_-^{-1} \\
    -\omega_++\omega_- & J_+^*+J_-^* \\
  \end{array}
\right),\quad \mathbb{J}_2=\frac{1}{2}\left(
  \begin{array}{cc}
    -J_++J_-& \omega_+^{-1}+\omega_-^{-1} \\
    -\omega_+-\omega_- & J_+^*-J_-^* \\
  \end{array}
\right).\]
Let $\beta_1:=-\frac{1}{2}(J_+-J_-)g^{-1}$ and $\beta_2:=-\frac{1}{2}(J_++J_-)g^{-1}$. These are actually Poisson structures associated to $\mathbb{J}_{1}$ and $\mathbb{J}_{2}$ respectively.
\section{Hamiltonian Generalized K$\ddot{a}$hler manifolds}\label{Hami}
We first recall the generalized K$\ddot{a}$hler reduction procedure developed in \cite{LT}. Our formulation here is, however, greatly influenced by the works in \cite{BCG1} \cite{BCG2} \cite{Ca}. In particular, we stick to the metric splitting, which exists naturally on a generalized K$\ddot{a}$hler manifold $M$.

Let a compact connected Lie group $G$ act on $M$ from the left in the extended manner, preserving the generalized K$\ddot{a}$hler structure on $M$ and consequently the metric splitting. The notion of (generalized) moment map can be defined in the context of generalized complex manifolds. However, a generalized K$\ddot{a}$hler manifold consists of two generalized complex structures $(\mathbb{J}_1, \mathbb{J}_2)$. As a convention, when referring to a moment map, we always mean it is associated to $\mathbb{J}_2$.
\begin{definition}Let $M$ be a generalized K$\ddot{a}$hler manifold carrying an extended $G$-action $\varphi$ preserving the underlying generalized K$\ddot{a}$hler structure. An equivariant map $\mu: M\rightarrow \mathfrak{g}^*$ is called a moment map, if
\begin{equation}\mathbb{J}_2(X_\varsigma+\xi_\varsigma-\sqrt{-1}d\mu_\varsigma)=\sqrt{-1}(X_\varsigma+\xi_\varsigma-\sqrt{-1}d\mu_\varsigma)\label{gmm}\end{equation}
for any $\varsigma\in \mathfrak{g}$. If this happens, we call $M$ a Hamiltonian $G$-generalized K$\ddot{a}$hler manifold.
\end{definition}
In contrast with ordinary K$\ddot{a}$hler reduction, $\mathbb{J}_2$ plays the role of a symplectic structure in K$\ddot{a}$hler reduction and it is expected that $\mathbb{J}_1$ should act like a complex structure in K$\ddot{a}$hler reduction. Note that Eq.~(\ref{gmm}) is precisely
\begin{equation}\mathbb{J}_2(X_\varsigma+\xi_\varsigma)=d\mu_\varsigma.\label{gmm2}\end{equation}

The following algebraic calculation actually has already appeared in \cite{BL} \cite{Ni}. We include it here only because it provides some motivations for our later considerations.
\begin{lemma}\label{lemma}
In terms of the bihermitian data, Eq.~(\ref{gmm}) is equivalent to the following two equations:
\begin{equation}
J_+X_\varsigma^+=J_-X_\varsigma^-=-g^{-1}d\mu_\varsigma,\label{cent}
\end{equation}
where $X_\varsigma^\pm=X_\varsigma\pm g^{-1}\xi_\varsigma$,
\end{lemma}
\begin{proof}
The equation \[\mathbb{J}_2\left(
         \begin{array}{c}
           X_\varsigma \\
           \xi_\varsigma \\
         \end{array}
       \right)=\frac{1}{2}\left(
  \begin{array}{cc}
    -J_++J_-& \omega_+^{-1}+\omega_-^{-1} \\
    -\omega_+-\omega_- & J_+^*-J_-^* \\
  \end{array}
\right)\left(
         \begin{array}{c}
           X_\varsigma \\
           \xi_\varsigma \\
         \end{array}
       \right)=\left(
                 \begin{array}{c}
                   0 \\
                   d\mu_\varsigma \\
                 \end{array}
               \right)
\]
 written in components is
 \[(-J_++J_-)X_\varsigma-(J_++J_-)g^{-1}\xi_\varsigma=0,\]
and
\[-(\omega_+(X_\varsigma)+\omega_-(X_\varsigma))+(J_+^*-J_-^*)\xi_\varsigma=2d\mu_\varsigma.\]
The first is precisely $J_+X_\varsigma^+=J_-X_\varsigma^-$. Substituting this into the second leads to $J_+X_\varsigma^+=-g^{-1}d\mu_\varsigma$.
\end{proof}
From the lemma, we have $X_\varsigma=-\beta_2(d\mu_\varsigma)$ and $g^{-1}\xi_\varsigma=-\beta_1(d\mu_\varsigma)$. That's to say, $X_\varsigma$ and $g^{-1}\xi_\varsigma$ are precisely the Hamiltonian vector field of $\mu_\varsigma$ w.r.t. $\beta_1$ and $\beta_2$ respectively. The lemma, together with Eq.~(\ref{gmm2}), is precisely the analogue of the fact in Hamiltonian equivariant K$\ddot{a}$hler geometry that $X_\varsigma$ is the Hamiltonian vector field of $\mu_\varsigma$ and $JX_\varsigma$ is minus the gradient vector field of $\mu_\varsigma$. \emph{For later convenience, we denote $J_+X_\varsigma^+$ by $Y_\varsigma$.}

If $G$ acts freely on $\mu^{-1}(0)$, then $\mu^{-1}(0)/G$ acquires a generalized K$\ddot{a}$hler structure through the general reduction theory. Let $K_0$ be the subbundle of $E|_{\mu^{-1}(0)}$ generated by $\varphi(\mathfrak{g})$ and $d\mu_\varsigma$, $\varsigma\in \mathfrak{g}$. Then $K_0$ is isotropic and $\frac{K_0^\bot}{K_0}/G$ is the reduced Courant algebroid over $\mu^{-1}(0)/G$. Since $K_0$ is $\mathbb{J}_2$-invariant, there is a natural complex structure $\mathbb{J}_2^{red}$ on $\frac{K_0^\bot}{K_0}/G$. This is the reduced version of $\mathbb{J}_2$. There is another way to see clearly how the two generalized complex structures descend simultaneously in a compatible way: $\frac{K_0^\bot}{K_0}$ is naturally isomorphic to $K_0^\bot\cap \mathcal{G}K_0^\bot$, where $\mathcal{G}=-\mathbb{J}_1\mathbb{J}_2$. Since the latter is both $\mathcal{G}$-invariant and $\mathbb{J}_2$-invariant, it is also $\mathbb{J}_1$-invariant and therefore one just has to restrict $\mathbb{J}_1$ and $\mathbb{J}_2$ to $K_0^\bot\cap \mathcal{G}K_0^\bot$ to find the reduced structures $\mathbb{J}_1^{red}$ and $\mathbb{J}_2^{red}$.

Now we turn to the problem of complexifying the extended $G$-action on $M$. At the infinitesimal level, the naive choices $J_\pm X_\varsigma$ motivated by the K$\ddot{a}$hler case won't work because generally $J_+X_\varsigma\neq J_-X_\varsigma$. But the above lemma provides an alternative choice: We can use $Y_\varsigma=J_+X_\varsigma^+$ ($=J_-X_\varsigma^-$) instead of $J_+X_\varsigma$ or $J_-X_\varsigma$. This is justified by the fact that
\begin{equation}\mathbb{J}_1 (X_\varsigma+\xi_\varsigma)=-\mathbb{J}_1\mathbb{J}_2 d\mu_\varsigma=g^{-1}d\mu_\varsigma=-J_+X_\varsigma^+=-Y_\varsigma.\end{equation}
\begin{definition}The map $\varphi_{\mathbb{C}}: \mathfrak{g}_{\mathbb{C}}\rightarrow \Gamma(E)$ defined by
\begin{equation}\varphi_{\mathbb{C}}(u+\sqrt{-1}v)=\varphi(u)-\mathbb{J}_1\varphi(v)\label{comp}\end{equation}
is called the pre-complexification of $\varphi$.
\end{definition}
\emph{Remark}. The minus sign in (\ref{comp}) is to keep accordance with the specified case of Hamiltonian equivariant K$\ddot{a}$hler manifolds where $J_+=J_-$.
\begin{proposition}\label{pre}The pre-complexification $\varphi_\mathbb{C}$ of $\varphi$ defined above preserves $\mathbb{J}_1$.
\end{proposition}
\begin{proof}It suffices to verify that $\mathbb{J}_1(X_\varsigma+\xi_\varsigma)$ preserves $\mathbb{J}_1$: For any $A\in \Gamma(\mathbb{T})$,
\begin{eqnarray*}[\mathbb{J}_1(X_\varsigma+\xi_\varsigma), \mathbb{J}_1 A]_H&=&\mathbb{J}_1[X_\varsigma+\xi_\varsigma, \mathbb{J}_1 A]_H+\mathbb{J}_1[\mathbb{J}_1(X_\varsigma+\xi_\varsigma), A]_H+[X_\varsigma+\xi_\varsigma, A]_H\\
&=&\mathbb{J}_1[\mathbb{J}_1(X_\varsigma+\xi_\varsigma), A]_H
\end{eqnarray*}
where Eq.(\ref{inteJ}) and the fact $X_\varsigma+\xi_\varsigma$ preserves $\mathbb{J}_1$ are used.
\end{proof}
\emph{Remark}. The proposition is precisely the analogue of the fact in the K$\ddot{a}$hler case that the vector field $JX_\varsigma$ preserves the complex structure $J$.

Up to now, it seems that we are on the right way to complexifying the $G$-action. However, we come across some difficulty--The map $\varphi_\mathbb{C}$ is not necessarily a Lie algebra homomorphism. Let $S: \mathfrak{g}\times \mathfrak{g}\rightarrow C^\infty(M)$ be defined by $S(\varsigma, \zeta)=(\mathbb{J}_1\varphi(\varsigma), \varphi(\zeta))$ for $\varsigma,\zeta\in \mathfrak{g}$. Obviously $S(\varsigma, \zeta)=-S(\zeta, \varsigma)$. We have the following characterization of $S$.
\begin{proposition}
$S(\varsigma,\zeta)=\{\mu_\varsigma, \mu_\zeta\}_{\beta_1}$, i.e. the Poisson bracket of $\mu_\varsigma, \mu_\zeta$ w.r.t. the Poisson structure $\beta_1$.
\end{proposition}
\begin{proof}
\[S(\varsigma,\zeta)=(\mathbb{J}_1(X_\varsigma+\xi_\varsigma), X_\zeta+\xi_\zeta)=(\mathbb{J}_1\mathbb{J}_2 d\mu_\varsigma, \mathbb{J}_2d\mu_\zeta)=(\mathbb{J}_1 d\mu_\varsigma, d\mu_\zeta)=\{\mu_\varsigma, \mu_\zeta\}_{\beta_1}.\]
\end{proof}
 The appearance of $S$ is a totally new phenomenon compared with the K$\ddot{a}$hler case where $\beta_1\equiv 0$. $S$ is generally the obstruction for $\varphi_\mathbb{C}$ to be a Lie algebra homomorphism. This can be seen from the following direct computations:
\[[X_\varsigma+\xi_\varsigma, \mathbb{J}_1(X_\zeta+\xi_\zeta)]_H=\mathbb{J}_1[X_\varsigma+\xi_\varsigma, X_\zeta+\xi_\zeta]_H,\]
\[[\mathbb{J}_1(X_\varsigma+\xi_\varsigma), X_\zeta+\xi_\zeta)]_H=\mathbb{J}_1[X_\varsigma+\xi_\varsigma, X_\zeta+\xi_\zeta]_H+dS(\varsigma,\zeta),\]
\[[\mathbb{J}_1(X_\varsigma+\xi_\varsigma), \mathbb{J}_1(X_\zeta+\xi_\zeta)]_H=-[X_\varsigma+\xi_\varsigma, X_\zeta+\xi_\zeta]_H+\mathbb{J}_1 d S(\varsigma,\zeta),\]
where Eq.~(\ref{nonsk}) and Eq.~(\ref{inteJ}) are used. Note that $\{X_\varsigma, Y_\varsigma\}_{\varsigma\in \mathfrak{g}}$ spans a smooth distribution in the sense of Sussmann \footnote{Such a distribution is not necessarily of constant rank.} \cite{Sus}. This distribution is not necessarily integrable. However, we still have
\begin{proposition}\label{suss}The smooth distribution $D$ spanned by $\{X_\varsigma, Y_\varsigma\}_{\varsigma\in \mathfrak{g}}$ and the image of $\beta_1: T^*\rightarrow T$ is integrable.
\end{proposition}
\begin{proof}Since the Poisson structure $\beta_1$ is $G$-invariant, we only need to prove for any smooth function $f$ the Lie bracket of $Y_\varsigma$ and $\beta_1(df)$ lies in $D$. Note that by Eq.~(\ref{inteJ}), we have
\begin{eqnarray*}[\mathbb{J}_1(X_\varsigma+\xi_\varsigma), \mathbb{J}_1(df)]_H&=&\mathbb{J}_1[X_\varsigma+\xi_\varsigma, \mathbb{J}_1(df)]_H+\mathbb{J}_1[\mathbb{J}_1(X_\varsigma+\xi_\varsigma), df]_H\\
&+&[X_\varsigma+\xi_\varsigma, df]_H\\
&=&-[X_\varsigma+\xi_\varsigma, df]_H+\mathbb{J}_1[\mathbb{J}_1(X_\varsigma+\xi_\varsigma), df]_H+[X_\varsigma+\xi_\varsigma, df]_H\\
&=&\mathbb{J}_1[\mathbb{J}_1(X_\varsigma+\xi_\varsigma), df]_H\\
&=&-\mathbb{J}_1(dL_{Y_\varsigma }f),
\end{eqnarray*}
where we have used the fact that $\mathbb{J}_1$ is preserved by the extended $G$-action. Therefore,
\[\pi([\mathbb{J}_1(X_\varsigma+\xi_\varsigma), \mathbb{J}_1(df)]_H)=-\beta_1(dL_{Y_\varsigma }f).\]
However, the left hand side of the above equality is nothing else but $-[Y_\varsigma, \beta_1(df)]$. Our conclusion thus holds.
\end{proof}
\emph{Remark}. In particular, we have proved that $\beta_1$ is preserved by the vector fields $Y_\varsigma$. We feel the following proposition will be of value if one is to consider the interaction between the group action and $\beta_1$.
\begin{proposition}If either $X_\varsigma$ or $Y_\varsigma$ is tangent to a sympletic leaf $\mathcal{L}$ of $\beta_1$ at a point $x\in \mathcal{L}$, so is the other.
\end{proposition}
\begin{proof}If $X_\varsigma$ is tangent to $\mathcal{L}$ at $x$, then $X_\varsigma=\beta_1(df)$ at $x$ for some function $f$ defined around $x$. Therefore, at $x$
\[X_\varsigma+\xi_\varsigma=\mathbb{J}_1(df)+\tau\]
for some 1-form $\tau$. So
\[Y_\varsigma=-\mathbb{J}_1(X_\varsigma+\xi_\varsigma)=df-\mathbb{J}_1(\tau).\]
This shows that $Y_\varsigma=-\beta_1(\tau)$, i.e. $Y_\varsigma$ is also tangent to $\mathcal{L}$ at $x$. The converse can be proved similarly.
\end{proof}
\section{Strong Hamiltonian actions}\label{stro}
We have found in the previous section that the Poisson structure $\beta_1$ appears as the obstruction for complexfying the extended $\mathfrak{g}$-action. To proceed further, we are forced to introduce the following notion of \emph{strong Hamiltonian action}.
\begin{definition}\label{str}
If $M$ is a Hamiltonian $G$-generalized K$\ddot{a}$hler manifold, the $G$-action is called strong Hamiltonian if the map $S$  defined in \S ~\ref{Hami} vanishes.
\end{definition}
The definition has a direct consequence that in the strong Hamiltonian case, the underlying Poisson structure $\beta_1$ is preserved by the $\mathfrak{g}_\mathbb{C}$-action.
\begin{example}An ordinary Hamiltonian $G$-K$\ddot{a}$hler manifold viewed as a Hamiltonian $G$-generalized K$\ddot{a}$hler manifold is, of course, strong Hamiltonian.
\end{example}
The first nontrivial examples of strong Hamiltonian generalized K$\ddot{a}$hler manifolds are provided by the following
\begin{proposition}All Hamiltonian $S^1$-generalized K$\ddot{a}$hler manifolds are strong Hamiltonian.
\end{proposition}
\begin{proof}This is obvious due to dimensional reason and the fact that $S$ is skew-symmetric.
\end{proof}
\emph{Remark}. Since Cartan subgroups are a basic ingredient of compact connected Lie groups, strong Hamiltonian $S^1$-actions naturally arise on a Hamiltonian $G$-generalized K$\ddot{a}$hler manifold.
\begin{example}Let $(M,g, I, J, K)$ be a hyperK$\ddot{a}$hler structure, and $\omega_I, \omega_J, \omega_K$ be the associated K$\ddot{a}$hler forms. According to the observation in \cite{Gu00}, a generalized K$\ddot{a}$hler structure can be constructed as follows:
\[\mathbb{J}_1=\left(
                 \begin{array}{cc}
                   1 & 0 \\
                   -\omega_K & 1 \\
                 \end{array}
               \right)\left(
                        \begin{array}{cc}
                          0 & \frac{1}{2}(\omega_I^{-1}+\omega_J^{-1}) \\
                          -\omega_I-\omega_J & 0 \\
                        \end{array}
                      \right)\left(
                 \begin{array}{cc}
                   1 & 0 \\
                   \omega_K & 1 \\
                 \end{array}
               \right)
\]
\[\mathbb{J}_2=\left(
                 \begin{array}{cc}
                   1 & 0 \\
                   \omega_K & 1 \\
                 \end{array}
               \right)\left(
                        \begin{array}{cc}
                          0 & \frac{1}{2}(\omega_I^{-1}-\omega_J^{-1}) \\
                          -\omega_I+\omega_J & 0 \\
                        \end{array}
                      \right)\left(
                 \begin{array}{cc}
                   1 & 0 \\
                   -\omega_K & 1 \\
                 \end{array}
               \right).
\]
Note that we are already in the metric splitting. Suppose $M$ carries a Hamiltonian $\mathbb{T}^k$(torus of dimension $k$)-action with moment map $(\mu^I, \mu^J, \mu^K): M\rightarrow \mathfrak{t}^*\oplus\mathfrak{t}^*\oplus\mathfrak{t}^*$, where $\mathfrak{t}$ is the Lie algebra of $\mathbb{T}^k$. Then for $\varsigma\in \mathfrak{t}$, we have $\iota_{X_\varsigma}\omega_I=-d\mu_{\varsigma}^I$,
$\iota_{X_\varsigma}\omega_J=-d\mu_{\varsigma}^J$ and $\iota_{\varsigma}\omega_K=-d\mu_{\varsigma}^K$. According to computations in \cite{LT}, $\mu:=\mu^I-\mu^J$ is a moment map associated to $\mathbb{J}_2$. This $\mathbb{T}^k$-action is actually strong Hamiltonian: Note that $\omega_I^{-1}d\mu_\varsigma=-KX_\varsigma-X_\varsigma$ and $\omega_J^{-1}d\mu_\varsigma=-KX_\varsigma+X_\varsigma$. So $$\beta_1(d\mu_\varsigma)=\frac{1}{2}(\omega_I^{-1}+\omega_J^{-1})d\mu_\varsigma=-KX_\varsigma.$$ We thus have
\begin{eqnarray*}\{\mu_\varsigma, \mu_\zeta\}_{\beta_1}&=&-d\mu_\zeta(KX_\varsigma)=d\mu_{\zeta}^J(KX_\varsigma)-d\mu_{\zeta}^I(KX_\varsigma)\\
&=&-\omega_J(X_\zeta, KX_\varsigma)+\omega_I(X_\zeta, KX_\varsigma)\\
&=&g(KJX_\zeta,X_\varsigma)-g(KIX_\zeta,X_\varsigma)\\
&=&-g(IX_\zeta,X_\varsigma)-g(JX_\zeta,X_\varsigma)\\
&=&\omega_I(X_\varsigma,X_\zeta)+\omega_J(X_\varsigma,X_\zeta)\\
&=&X_\varsigma\mu_{\zeta}^I+X_\varsigma\mu_{\zeta}^J\\
&=&0.
\end{eqnarray*}
The last equality is due to the fact that $\mu^I$ and $\mu^J$ are actually $\mathbb{T}^k$-invariant.
\end{example}
\begin{example}Let us first construct an $SU(N-3)$-invariant generalized K$\ddot{a}$hler structure on $M=\mathbb{C}P^N$ where $N>3$. This is adapted from a generalized K$\ddot{a}$hler structure on $\mathbb{C}P^2$ in \cite{BCG1} as follows. Let $\mathbb{\mathcal{J}}_1$, $\mathbb{\mathcal{J}}_2$ be the canonical generalized complex structures on $\mathbb{C}^{N+1}$ associated to the canonical complex and symplectic structures. $S^1$ acts on $\mathbb{C}^{N+1}$ by scaling, with a moment map $\mu_0(z)=|z|^2-1$. One chooses an $S^1$-invariant deformation $\varepsilon$ of $\mathcal{J}_1$ while keeping $\mathcal{J}_2$ fixed. For instance, we choose
\[\varepsilon=\frac{1}{2}z_0^2(\partial_{z_1}+\frac{1}{2}d\bar{z}_1)\wedge (\partial_{z_2}-\frac{1}{2}d\bar{z}_2). \]
Then $\mathcal{J}_1$ is deformed to another generalized complex structure $\mathcal{J}_1^\varepsilon$ away from the cylinder $|z_0|=\sqrt{2}$. A pure spinor\footnote{We haven't reviewed the pure spinor description of a generalized complex structure in this paper, for this see \cite{Gu00}.} of $\mathcal{J}_1^\varepsilon$ is
\[\varphi^\varepsilon=e^\varepsilon\cdot (\Omega_1\wedge \Omega_2)=(e^\varepsilon\cdot \Omega_1)\wedge \Omega_2,\]
where $\Omega_1=dz_0dz_1dz_2$ and $\Omega_2=dz_3dz_4\cdots dz_N$. Then the pair $(\mathcal{J}_1^\varepsilon, \mathcal{J}_2)$ is an $S^1$-invariant generalized K$\ddot{a}$hler structure, and by the reduction procedure developed in \cite{LT} $\mathbb{C}P^N$ acquires a generalized K$\ddot{a}$hler structure $(\mathbb{J}_1, \mathbb{J}_2)$. Actually, the reduction of $\mathcal{J}_2$ is precisely the classical Marsden-Weinstein reduction, giving rise to the Fubini-Study symplectic form. Note that the standard action of $SU(N+1)$ on $\mathbb{C}^{N+1}$ commutes with the $S^1$-action and preserves $\mathcal{J}_2$. Thus the $SU(N+1)$-action descends to $\mathbb{C}P^N$ and preserves the Fubini-Study form. This Fubini-Study action is Hamiltonian with a canonical moment map.

Let us describe a strong Hamiltonian $SU(N-3)$-action on the generalized K$\ddot{a}$hler manifold $(\mathbb{C}P^N, \mathbb{J}_1, \mathbb{J}_2)$. $SU(N-3)$, as a subgroup of $SU(N+1)$, acts trivially on the first three components of $\mathbb{C}^{N+1}$ and nontrivially on the remainder components. By our construction, $\varepsilon$ is invariant under this $SU(N-3)$-action and therefore the pair $(\mathbb{J}_1, \mathbb{J}_2)$ is also invariant. The moment map $\mu$ of this $SU(N-3)$-action is just the restriction of the Fubini-Study moment map on the Lie algebra of $SU(N-3)$. By our construction, only the tranverse complex coordinates induced from $\mathbb{J}_1$ are involved in the moment map. The components of $\mu$ are thus Casimir functions w.r.t. $\beta_1$ and hence the $SU(N-3)$-action is strong Hamiltonian.
\end{example}

We will give more nontrivial examples in our future work, and here only turn to the general question of complexifying the extended $G$-action.
\begin{proposition}If $M$ is a strong Hamiltonian $G$-generalized K$\ddot{a}$hler manifold, the pre-complexification $\varphi_\mathbb{C}$ of $\varphi$ is actually an extended $\mathfrak{g}_{\mathbb{C}}$-action preserving $\mathbb{J}_1$.
\end{proposition}
\begin{proof}By definition of Hamiltonian actions, the $\mathfrak{g}$-action is isotropic trivially extended. That this continues to hold for $\varphi_\mathbb{C}$ is a direct result of the fact $S(\varsigma,\zeta)=(\mathbb{J}_1\varphi(\varsigma), \varphi(\zeta))$. Since the obstruction $S$ vanishes, $\varphi_\mathbb{C}$ is thus a Lie algebra homomorphism. Prop.~\ref{pre} shows this extended $\mathfrak{g}_\mathbb{C}$-action preserves $\mathbb{J}_1$.
\end{proof}
The next question we shall tackle is whether the extended $\mathfrak{g}_\mathbb{C}$-action can be integrated to a $G^\mathbb{C}$-action. We content ourselves with the following result.
\begin{theorem}\label{inte}If $M$ is a strong Hamiltonian $G$-generalized K$\ddot{a}$hler compact manifold, then the extended $\mathfrak{g}_\mathbb{C}$-action defined in Def. \ref{str} integrates to a unique extended $G^\mathbb{C}$-action.
\end{theorem}
\begin{proof}Let $\tilde{G}$ be the universal cover of $G$. The complexification $\varphi_\mathbb{C}$ actually gives rise to an infinitesimal Lie algebra action of $\mathfrak{g}_\mathbb{C}$ on the total space of the Courant algebroid $E$. Recall the fact that a $\mathfrak{g}_\mathbb{C}$-action on a manifold $N$ integrates to a $\tilde{G}^\mathbb{C}$-action if and only if each vector field on $N$ generated by an element in $\mathfrak{g}_\mathbb{C}$ is complete \cite[Appendix B]{GSK}. Since we already have the $G$-action on $E$, to prove the theorem, it suffices to check the infinitesimal automorphism $(Y_\varsigma, -\iota_{Y_\varsigma}H)$, regarded as a vector field on $E$, generates a global flow. Let $\psi_t$ be the flow on $M$ generated by $Y_\varsigma$. It is of course global since $M$ is compact. Define a one-parameter 2-form on $M$
\[B_t:=-\int_0^t\psi_{-\tau} ^*(\iota_{Y_\varsigma}H)d\tau,\quad t\in \mathbb{R}.\]
Then it is routine to check that the pair $(\psi_t, B_t)$ acting on $E$ gives the desired global flow.

There should be a discrete subgroup $\Pi$ of $\tilde{G}$ contained in the center of $\tilde{G}$ such that
\[G=\tilde{G}/\Pi, \quad G^\mathbb{C}=\tilde{G}^\mathbb{C}/\Pi.\]
Since when restricted on $\tilde{G}$ the $\tilde{G}^\mathbb{C}$-action factors through $\Pi$,
we actually have an extended $G^\mathbb{C}$-action. The uniqueness of this extension is determined by the infinitesimal action.

\end{proof}
\emph{Remark}. In ordinary equivariant K$\ddot{a}$hler geometry, the theorem has a more direct proof (see for example \cite{GS1}), which uses the well-known fact that the automorphism group of a compact complex manifold is a Lie transformation group. However, this won't work for the present setting if one just recalls the equally well-known fact that the automorphism group of a compact symplectic manifold is not a Lie transformation group. It is also remarkable here that the underlying $G^\mathbb{C}$-action in general will preserve neither $J_+$ nor $J_-$, and thus is not holomorphic w.r.t. either of the two.
\section{GIT quotients and generalized holomorphic structures}\label{GIT}
Throughout this section, let $M$ be a compact $G$-generalized K$\ddot{a}$hler manifold. To have a good quotient, we assume that we are in the strong Hamiltonian case, and that 0 is a regular value of $\mu$.\footnote{Note that $X_\varsigma=-\beta_2(d\mu_\varsigma)$. In the K$\ddot{a}$hler case, $\beta_2$ is invertible and therefore that (i) 0 is a regular value of $\mu$ is equivalent to that (ii) $G$ acts locally freely on $\mu^{-1}(0)$. This equivalence still holds here but needs more technical arguments, which can be found in \cite{Ni}.} Consequently $0$ is a regular value of $\mu$, the strong Hamiltonian $G$-action on $M$ can be complexified and we can talk about the extended $G^{\mathbb{C}}$-action and $G^{\mathbb{C}}$-orbits in $M$.

Define $M_s:=\{g x|g\in G^\mathbb{C}, x\in \mu^{-1}(0)\}$. From a GIT viewpoint, points in $M_s$ can reasonably be called stable and $M_s$ be called the stable locus.
\begin{proposition}\label{stableloci}$M_s$ is an open set of $M$.
\end{proposition}
\begin{proof}
If $M_s$ contains an open neighbourhood $U$ of $\mu^{-1}(0)$, then $M_s=\cup_{g\in G^\mathbb{C}}g U$ and is thus an open set of $M$. Let $\{\varsigma_a\}$ be a basis of $\mathfrak{g}$ and $X_a$ the vector field generated by $\varsigma_a$. Since $\mu^{-1}(0)$ has codimension $\textup{dim} \mathfrak{g}$ in $M$, we only need to prove the vector fields $Y_a=J_+X_a^+$ are linearly independent over $\mu^{-1}(0)$ and orthogonal to $T\mu^{-1}(0)$ in $T$.

 If $x\in \mu^{-1}(0)$ and $\sum_af^aY_a=0$ at $x$ for some constants $f^a$, then $\sum_af^aX_a^+=\sum_af^aX_a^-=0$ at $x$ due to the fact that $Y_a=J_+X_a^+=J_-X_a^-$. Therefore we have $\sum_af^aX_a=0$ and $\sum_a f^a \varsigma_a$ lies in the Lie algebra of the stabilizer $G_x (\subset G)$ of $x$. All $f^a$ should be zero because $G$ acts locally freely on $\mu^{-1}(0)$. Therefore, vector fields $Y_a$'s are linearly independent pointwise over $\mu^{-1}(0)$.
If $v\in T_x\mu^{-1}(0)$, then at $x$ $g(Y_a, v)=-d\mu_{\varsigma_a}(v)=0$ by definition. This completes the proof.
\end{proof}
\begin{proposition}\label{stablizer}
The stabilizer $G^\mathbb{C}_x$ in $G^\mathbb{C}$ of a point $x\in M_s$ is finite. If $x\in \mu^{-1}(0)$, then $G_x^\mathbb{C}=G_x$; in particular, if $G$ acts freely on $\mu^{-1}(0)$, then $G^\mathbb{C}$ acts freely on $M_s$.
\end{proposition}
\begin{proof}By the well-known Cartan decomposition $G^\mathbb{C}=PG$ where $P$ is diffeomorphic to $\sqrt{-1}\mathfrak{g}$ via the exponential map, and by the fact that vector fields such as $Y_\varsigma$ generate the action of $\sqrt{-1}\mathfrak{g}$, the proof of Prop.~\ref{stableloci} already implies that the stabilizer of $x\in \mu^{-1}(0)$ is finite. Since each point $y\in M_s$ lies in a $G^\mathbb{C}$-orbit which intersects $\mu^{-1}(0)$ at some point $x$, the stabilizer of $y$ differs from the stabilizer of $x$ only by conjugation and thus is also finite.

For $x\in \mu^{-1}(0)$, obviously we have $G_x\subseteq G^\mathbb{C}_x$. 
Let $g\in G_x^\mathbb{C}\backslash G_x$. Due to Cartan decomposition, we can write it as $g=\exp(\sqrt{-1}\varsigma)k$, where $0\neq\varsigma \in \mathfrak{g}$ and $k\in G$. Consider the curve $s(t)=\exp(-\sqrt{-1}t\varsigma) kx$ ($t\in \mathbb{R}$) in $M_s$ and the function $h(t)=\mu_\varsigma(s(t))$. Note that
\[\frac{dh}{dt}=d\mu_\varsigma(Y_\varsigma)=-g(Y_\varsigma,Y_\varsigma).\]
Since $G^\mathbb{C}$ acts locally freely on $M_s$, $Y_\varsigma$ vanishes nowhere on $M_s$ and thus $\frac{dh}{dt}<0$. Now that $kx\in \mu^{-1}(0)$, we have
\[0=\mu_\varsigma(x)-\mu_\varsigma(kx)=h(1)-h(0)=\int_0^1 \frac{dh}{dt} dt<0,\]
which is a contradiction! Therefore, we must have $G_x^\mathbb{C}=G_x$. The rest of this proposition is a direct result of this equality.
\end{proof}
\begin{proposition}\label{Gorbit}
For any point $x\in \mu^{-1}(0)$, $G^\mathbb{C} x \cap \mu^{-1}(0)=G x$.
\end{proposition}
\begin{proof}
It suffices to prove $G^\mathbb{C} x \cap \mu^{-1}(0) \subseteq Gx$. If $y=g x\in \mu^{-1}(0)$ for some $g\in G^\mathbb{C}$, without loss of generality we can assume $g=\exp(\sqrt{-1}\varsigma)$. Consider the similar curve $s(t)$ and the function $h(t)$ in the proof of Prop.~\ref{stablizer} ( $k=e$). Note that
\[0=\mu_\varsigma(y)-\mu_\varsigma(x)=h(1)-h(0)=-\int_0^1 g(Y_\varsigma,Y_\varsigma)\leq 0,\]
and equality holds iff $Y_\varsigma=0$ along $s(t)$. This immediately implies that $X_\varsigma=0$ along $s(t)$. Since $G$ acts locally freely on $M_s$, we must have $\varsigma=0$ and thus $y=x\in G x$.
\end{proof}
\begin{proposition}\label{Hauss}If $x,y\in \mu^{-1}(0)$ and $x\notin G y$, then there are disjoint $G^\mathbb{C}$-invariant neighbourhoods of $x$ and $y$ in $M$.
\end{proposition}
\begin{proof}
We shall follow closely the ideas of Kirwan in \cite{Kir}. There is a compact $G$-invariant neighbourhood $W$ of $x$ in $\mu^{-1}(0)$ not containing $y$, because $G$ is compact and $x\notin G y$. By the proof of Prop.~\ref{stableloci}, $\exp (\sqrt{-1}\mathfrak{g}) W$ is a neighbourhood of $x\in M$. It suffices to prove that $y\notin \overline{\exp (\sqrt{-1}\mathfrak{g})W}$ for $\exp (\sqrt{-1}\mathfrak{g}) W$ and $M\setminus\overline{\exp (\sqrt{-1}\mathfrak{g})W}$ will be the neighbourhoods we need.

Choose an invariant metric on $\mathfrak{g}$ and use it to identify $\mathfrak{g}^*$ with $\mathfrak{\varsigma}$. We use $\|\varsigma\|$ to denote the norm of $\varsigma\in \mathfrak{g}$. Let $$U=\{\exp(\sqrt{-1}\varsigma) z|\varsigma\in \mathfrak{g}, \|\varsigma\|\leq 1, z\in W\},$$
and define
\[\epsilon=\inf\{(X_\varsigma,X_\varsigma)(w)|w\in U, \varsigma\in \mathfrak{g}, \|\varsigma\|=1\}.\]
Since $G^\mathbb{C}$ acts locally freely on $M_s$ and $U$ is compact, we have $\epsilon>0$.

Suppose $z\in W$ and $\varsigma\in \mathfrak{g}$ such that $\|\varsigma\|=1$. Consider as in the proof of Prop.~\ref{stablizer} the function $h(t)=\mu_\varsigma(\exp(-\sqrt{-1}t\varsigma) z)$, $t\in \mathbb{R}$. Then $h'(t)\leq 0$; in particular, $h'(t)\leq -\epsilon$ for $x\in [0,1]$.
Since $h(0)=0$, for any $t\geq 0$ we have that $h(t)\leq 0$; moreover, for any $t\geq 1$ we have that
\begin{eqnarray*}\|\mu(\exp(\sqrt{-1}t\varsigma) z)\|&=&\|\mu(\exp(\sqrt{-1}t\varsigma) z)\|\cdot \|\varsigma\|\geq |\mu_\varsigma(\exp(\sqrt{-1}t\varsigma) z)|\\
&=&|h(t)|\geq |h(1)|=-\int_0^1 h'(t)dt\geq \epsilon. \end{eqnarray*}
So we have obtained the conclusion that $\|\mu(\exp(\sqrt{-1}\varsigma)z)\|\geq \epsilon$ for any $z\in W$ and $\|\varsigma\|\geq 1$.

Now we can prove $y\notin \overline{\exp (\sqrt{-1}\mathfrak{g}) W}$. If it is not so, there would be sequences $\varsigma_n\in \mathfrak{g}$, $v_n\in W$ such that $\lim_{n\rightarrow \infty}\exp(\sqrt{-1}\varsigma_n) v_n=y$.
By continuity, we would have
\[\lim_{n\rightarrow \infty}\|\mu(\exp(\sqrt{-1}\varsigma_n) v_n)\|=\lim_{n\rightarrow \infty}\|\mu(y)\|=0.\]
Thus for $n$ sufficiently large,
\[\|\mu(\exp(\sqrt{-1}\varsigma_n) v_n)\|< \epsilon.\]
Therefore, for $n$ sufficiently large $n$, $\exp(\sqrt{-1}\varsigma_n) v_n\in U$ and thus $\|\varsigma_n\|\leq 1$. By compactness, we have subsequences $\varsigma_{n_l}$ and $v_{n_l}$, which converge to some $\varsigma$ ($\|\varsigma\|\leq 1$) and $v\in W$ respectively. Thus $y=\exp(\sqrt{-1}\varsigma) v\in W\subset U$. But we already have $y\notin \exp (\sqrt{-1}\mathfrak{g}) W \supseteq U $. A contradiction!
\end{proof}
After the above preparation, we can give a theorem of Kempf-Ness type.
\begin{theorem}\label{kempf}
If $G$ acts locally freely on $\mu^{-1}(0)$, then the natural inclusion $i: \mu^{-1}(0)\hookrightarrow M_s$ induces a diffeomorphism
$\mu^{-1}(0)/G\cong M_s/G^{\mathbb{C}}$.
\end{theorem}
\begin{proof}Since $M_s=G^\mathbb{C} \mu^{-1}(0)$, the map $j: \mu^{-1}(0)/G\cong M_s/G^{\mathbb{C}}$ is obviously surjective. Prop.~ \ref{Gorbit} implies that $j$ is also injective. From Prop.~ \ref{Hauss}, we know $M_s/G^\mathbb{C}$ is a Hausdorff space. Therefore $j$ is a continuous bijection from a compact space to a Hausdorff space, and hence is a homeomorphism.

To see $j$ is actually a diffeomorphism, we only need to prove its smoothness locally. We only sketch a proof under the condition that $G$ acts freely on $\mu^{-1}(0)$. The case of locally free action will only involve some notational complexity.

By the famous Koszul-Palais slice theorem \cite[Appendix B]{GSK}, for $x\in \mu^{-1}(0)$, there is a slice $\mathcal{S}\subseteq \mu^{-1}(0)$ such that the map
\[\Theta: G\times \mathcal{S}\rightarrow \mu^{-1}(0),\quad (g, s)\mapsto gs\]
is a $G$-equivariant diffeomorphism from $G\times\mathcal{S}$ to its image in $\mu^{-1}(0)$. Thus $\mathcal{S}$ can be used as a coordinate chart around $[x]\in \mu^{-1}(0)/G$. $\Theta$ can be naturally $G^\mathbb{C}$-equivariantly extended:
\[\Theta^{\mathbb{C}}: G^{\mathbb{C}}\times \mathcal{S}\rightarrow M_s,\quad (g, s)\mapsto gs.\]
This is obviously a diffeomorphism from $G^{\mathbb{C}}\times \mathcal{S}$ to its image in $M_s$. Thus $\mathcal{S}$ can also be used as a coordinate chart around $[x]\in M_s/G^{\mathbb{C}}$. Therefore $j$ is the identity map in such a chart and consequently smooth.
\end{proof}

Following the convention of GIT, we denote $M_s/G^{\mathbb{C}}$ by $M//G^\mathbb{C}$. If additionally $G$ acts freely on $\mu^{-1}(0)$, then $M_s$ is open in $M$ and can be viewed as a principal $G^\mathbb{C}$-bundle over $M//G^\mathbb{C}$. Let $K$ be the subbundle of $E|_{M_s}$ generated by $\varphi(\mathfrak{g})$ and $\tilde{K}=K\oplus \mathbb{J}_1K$. Since $\tilde{K}$ is isotropic in $E|_{M_s}$ and the extended $G^\mathbb{C}$-action  preserves $\mathbb{J}_1$. We are in the situation of \cite[Thm.~5.2]{BCG1} and thus $\mathbb{J}_1$ descends to the quotient $M//G^\mathbb{C}$. Denote this generalized complex structure by $\mathbb{J}^\flat_1$.
\begin{theorem}\label{kempf2}
By the diffeomorphism $\mu^{-1}(0)/G\cong M//G^\mathbb{C}$, the generalized complex structure $\mathbb{J}_1^{red}$ coincides with $\mathbb{J}^\flat_1$.
\end{theorem}
\begin{proof}
By the general theory developed in \cite{BCG1} , the reduced Courant algebroid $E_{red}$ over $M//G^\mathbb{C}$ is $\frac{\tilde{K}^\bot}{\tilde{K}}/G^\mathbb{C}$. Since $\mathbb{J}_1$ preserves $\tilde{K}$, on $\frac{\tilde{K}^\bot}{\tilde{K}}$ there is a natural map $\tilde{\mathbb{J}}_1$, which gives rise to $\mathbb{J}^\flat_1$ after quotienting by $G^\mathbb{C}$. But $\mathbb{J}^\flat_1$ has another convenient description: Recall that $\mathcal{G}=-\mathbb{J}_1\mathbb{J}_2$ is the generalized metric. We have a natural isomorphism between $\frac{\tilde{K}^\bot}{\tilde{K}}$ and $\tilde{K}^\bot\cap \mathcal{G}\tilde{K}^\bot$, and by this identification $\tilde{\mathbb{J}}_1$ is precisely the restriction of $\mathbb{J}_1$ on $\tilde{K}^\bot\cap \mathcal{G}\tilde{K}^\bot$. To prove our theorem, we only need to check on $\mu^{-1}(0)$ it holds that
\begin{equation}K_0^\bot\cap \mathcal{G}K_0^\bot=\tilde{K}^\bot\cap \mathcal{G}\tilde{K}^\bot\label{iden},\end{equation}
where $K_0$ is the subbundle of $E|_{\mu^{-1}(0)}$ generated by $\varphi(\mathfrak{g})$ and $d\mu_\varsigma$, $\varsigma\in \mathfrak{g}$.
An element $Z+\eta$ in $K_0^\bot\cap \mathcal{G}K_0^\bot$ is characterized by the following equations:
\[\xi_\varsigma(Z)+\eta(X_\varsigma)=0,\quad d\mu_\varsigma(Z)=0,\quad g(X_\varsigma, Z)+g(\xi_\varsigma, \eta)=0,\quad g(\eta, d\mu_\varsigma)=0\]
for all $\varsigma\in \mathfrak{g}$
while an element $Z+\eta$ in $\tilde{K}^\bot\cap \mathcal{G}\tilde{K}^\bot$ is characterized by
\[\xi_\varsigma(Z)+\eta(X_\varsigma)=0,\quad \eta(Y_\varsigma)=0,\quad g(X_\varsigma, Z)+g(\xi_\varsigma, \eta)=0.\quad g(Z,Y_\varsigma)=0\]
for all $\varsigma\in \mathfrak{g}$.
Noting that $g^{-1}d\mu_\varsigma=-Y_\varsigma$, we find the two groups of equations on $\mu^{-1}(0)$ are actually the same. Therefore, Eq.~(\ref{iden}) does hold.
\end{proof}

Recall that a generalized complex structure has an important local invariant called type at a point, which is in fact the transverse complex dimension of the structure. Consider $\mathbb{J}_1$ and its reduction $\mathbb{J}_1^{red}$. We would like to know how the type $t(x)$ of a point $x\in M_s$ is related to the type $t([x])$ of its image $[x]\in M//G^\mathbb{C}$. Concerning this problem, there is a formula in \cite{LT}. In the following, we shall prove it in our context.\footnote{Lin-Tolman's formula applies to the more general case of Hamiltonian action. }
\begin{proposition}The two numbers $t([x])$ and $t(x)$ are related by
\begin{equation}t([x])=t(x)-\textup{dim}\mathfrak{g}+2\textup{dim}(\pi(L_1)\cap \pi(K_\mathbb{C}))_x.
\end{equation}
\end{proposition}
\begin{proof}Note that $\mathbb{J}_1^{red}$ at $[x]$ is actually modelled on the complex linear Dirac structure
\[\frac{L_1\cap \tilde{K}^\bot_\mathbb{C}+\tilde{K}_\mathbb{C}}{\tilde{K}_\mathbb{C}}\subset \tilde{K}^\bot_\mathbb{C}/\tilde{K}_\mathbb{C}\]
at $x$. Therefore, we have (to simplify notation, we omit the subscript $x$ in the following)
\begin{eqnarray*}T([x])&=&\frac{1}{2}\textup{dim}(\tilde{K}^\bot_\mathbb{C}/\tilde{K}_\mathbb{C})-\textup{dim}\pi(\frac{L_1\cap \tilde{K}^\bot_\mathbb{C}+\tilde{K}_\mathbb{C}}{\tilde{K}_\mathbb{C}})\\
&=& \textup{dim}M-2\textup{dim}\mathfrak{g}-\textup{dim}\frac{\pi(L_1\cap \tilde{K}^\bot_\mathbb{C})}{\pi(L_1\cap \tilde{K}_\mathbb{C}^\bot)\cap \pi(\tilde{K}_\mathbb{C})}
\\&=&\textup{dim}M-2\textup{dim}\mathfrak{g}-\textup{dim}\pi(L_1\cap \tilde{K}^\bot_\mathbb{C})+\textup{dim}(\pi(L_1\cap \tilde{K}_\mathbb{C}^\bot)\cap \pi(\tilde{K}_\mathbb{C})).\end{eqnarray*}
By Lemma 3.1 in \cite{LT},
\[\textup{dim}\pi(L_1\cap \tilde{K}^\bot_\mathbb{C})=\textup{dim}\pi(L_1+ K_\mathbb{C})-\textup{dim}\mathfrak{g}.\]
Since $\pi(L_1+ K_\mathbb{C})=\pi(L_1)+\pi(K_\mathbb{C})$, we have
\[\textup{dim}\pi(L_1+ K_\mathbb{C})=\textup{dim}\pi(L_1)+\textup{dim}\mathfrak{g}-\textup{dim}(\pi(L_1)\cap \pi(K_\mathbb{C})).\]
To prove the final result, we have to show
\[\textup{dim}(\pi(L_1\cap \tilde{K}_\mathbb{C}^\bot)\cap \pi(\tilde{K}_\mathbb{C}))=\textup{dim}\mathfrak{g}+\textup{dim}(\pi(L_1)\cap \pi(K_\mathbb{C})).\]
Let $K_a=L_1\cap\tilde{K}_\mathbb{C}$. Then $\tilde{K}_\mathbb{C}=K_a\oplus \overline{K_a}$ and $\pi(K_a)\subset \pi(L_1\cap \tilde{K}_\mathbb{C}^\bot)$. If $A\in K_a$ and $B\in \overline{K_a}$, for $\pi(A+B)$ to lie in $\pi(L_1\cap \tilde{K}_\mathbb{C}^\bot)$, the sufficient and necessary condition is $\pi(B)\in \pi(L_1\cap \tilde{K}_\mathbb{C}^\bot)$. But since $B\in \overline{K_a}\subset \overline{L_1}$, this happens iff $\pi(B)\in \pi(L_1)\cap \pi(\overline{L_1})$. Note that $\pi(B)=\sum_ic^i(X_{\varsigma_i}-\sqrt{-1}Y_{\varsigma_i})$ for a basic $\{\varsigma_i\}$ of $\mathfrak{g}$ and some constants $c^i$ and $\sum_ic^i(X_{\varsigma_i}+\sqrt{-1}Y_{\varsigma_i})$ lies in $\pi(L_1)$. Thus $\pi(B)\in \pi(L_1)\cap \pi(\overline{L_1})$ iff $\sum_ic^iX_{\varsigma_i}\in \pi(L_1)$. This completes the proof.
\end{proof}

If $G$ acts freely on $\mu^{-1}(0)$, $M_s$ is actually a generalized principal $G^\mathbb{C}$-bundle over $M//G^\mathbb{C}$ and $\mathbb{J}_1$ induces a generalized holomorphic structure over it.
\begin{theorem}\label{ghs}
$M_s$ viewed as a generalized principal $G^\mathbb{C}$-bundle on $M//G^\mathbb{C}$ is generalized holomorphic.
\end{theorem}
\begin{proof}
Let $K_a:=L_1\cap \tilde{K}_\mathbb{C}$ and $\mathcal{A}:=L_1\cap \tilde{K}^\bot_\mathbb{C}$. Then
\[K_a\oplus \overline{K_a}=\tilde{K}_\mathbb{C},\quad K_a\subset \mathcal{A} \subset \tilde{K}^\bot_\mathbb{C},  \]
and the rank of $\mathcal{A}$ is $\textup{dim}M-\textup{dim} \mathfrak{g}$.

We now show $\mathcal{A}$ is a generalized holomorphic structure. Obviously, $\mathcal{A}\oplus \overline{\mathcal{A}}\subset \tilde{K}^\bot_\mathbb{C}$. This inclusion is actually equality due to dimensional reason. Therefore the algebraic conditions in Def.~\ref{ghp} are all satisfied.

Let $A, B\in \Gamma(\mathcal{A})$ and $C\in \tilde{K}_\mathbb{C}$. Since $\mathbb{J}_1$ is integrable, $[A, B]_H\in \Gamma(L_1)$. Additionally,
\[( C, [A, B]_H)=\pi(A)( C, B)-([A, C]_H, B)
= ([C, A]_H, B).\]
But since the extended action of $G^\mathbb{C}$ preserves $\mathbb{J}_1$, we must have $[C, A]_H\in\Gamma(L_1)$. Therefore, $( C, [A, B]_H)=0$ and $[A, B]_H\in \Gamma(\mathcal{A})$, i.e. $\mathcal{A}$ is involutive.
\end{proof}
When $G$ acts freely on $\mu^{-1}(0)$, $\mu^{-1}(0)$ is then a principal $G$-bundle over $\mu^{-1}(0)/G$. Since nontrivial generalized holomorphic structures are not easy to construct, we would like to know when a vector bundle associated to a complex representation of $G$ acquires a generalized holomorphic structure, without complexifying the $G$-action. Our total argument then implies the following
\begin{theorem}In the strong Hamiltonian case if $G$ acts freely on $\mu^{-1}(0)$, then any associated complex vector bundle of the principal $G$-bundle $\mu^{-1}(0)\rightarrow \mu^{-1}(0)/G$ has a natural generalized holomorphic structure.
\end{theorem}
\begin{proof}
Let $(\phi, V)$ be a representation of $G$ in a complex vector space $V$ and $\phi_\mathbb{C}$ be the complexification of $\phi$. One simply complexifies the extended $G$-action on $M$. Then the (generalized) principal $G^\mathbb{C}$-bundle $M_s\rightarrow M//G^\mathbb{C}$ carries the natural generalized holomorphic structure. Therefore the associated vector bundle $V\times _{\phi_\mathbb{C}}M_s$ is a generalized holomorphic vector bundle. But as a vector bundle, $V\times_\phi \mu^{-1}(0)$ is the same as $V\times _{\phi_\mathbb{C}}M_s$.
\end{proof}
\section{Geometry on $G^\mathbb{C}$-orbits}\label{orbit}
As a concluding section, we consider the geometry of $G^\mathbb{C}$-orbits in $M_s$ of the former section, which arises from the ambient space.

Let $\mathfrak{i}:\mathcal{O}\hookrightarrow M_s$ be a $G^\mathbb{C}$-orbit in $M_s$. In the case of K$\ddot{a}$hler reduction, the K$\ddot{a}$hler structure on $M$ can be pulled back to $\mathcal{O}$ such that $\mathcal{O}$ is a $G$-invariant K$\ddot{a}$hler submanifold. Even more, the moment map $\mu$ can also be pulled-back to make the $G$-action on $\mathcal{O}$ Hamiltonian. We refer the interested readers to \cite{GS2} for details of this material.

At first glance, it would be expected that in the strong Hamiltonian generalized K$\ddot{a}$hler case, $\mathcal{O}$ may be a generalized K$\ddot{a}$hler submanifold in the sense that the two generalized complex structures $L_1$ and $L_2$, as complex Dirac structures, can both be pulled back onto $\mathcal{O}$ and induce a generalized K$\ddot{a}$hler structure on $\mathcal{O}$. However, this is indeed not the fact. As was observed in \cite{Vai}, a generalized K$\ddot{a}$hler submanifold of $M$ must be invariant under both $J_+$ and $J_-$. This cannot happen in general for the orbit $\mathcal{O}$ because $J_+X_\varsigma$ is not tangent to $\mathcal{O}$.

Astonishing is that $\mathcal{O}$ is actually again a Hamiltonian $G$-K$\ddot{a}$hler manifold in a natural way. In the following we describe how this structure arises.

$\mathcal{O}$ carries a natural complex structures $J_0$ defined by
\[J_0X_\varsigma=Y_\varsigma,\quad J_0Y_\varsigma=-X_\varsigma\]
at each $x\in \mathcal{O}$. $J_0$ is just induced from the canonical complex structure on $G^\mathbb{C}$. The metric on $\mathcal{O}$ cannot be the pull-back of $g$ since this is not Hermitian. We can define a new metric $g_0$ on $\mathcal{O}$ by letting
\[g_0(X_\varsigma, X_\zeta)=g_0(Y_\varsigma, Y_\zeta):=g(X_\varsigma, X_\zeta)+g(\xi_\varsigma, \xi_\zeta),\]
and
\[g_0(X_\varsigma, Y_\zeta)=g_0(Y_\zeta, X_\varsigma):=g(X_\varsigma, Y_\varsigma).\]
By definition $J_0$ and $g_0$ are obviously $G$-invariant. Additionally, the moment map $\mu$ can be pulled back to $\mathcal{O}$. By abuse of notation, we still use $\mu$ to denote this $G$-equivariant function from $\mathcal{O}$ to $\mathfrak{g}^*$.
\begin{theorem}\label{Ob}$(\mathcal{O}, J_0, g_0, \mu)$ is actually a Hamiltonian $G$-K$\ddot{a}$hler manifold.
\end{theorem}
\begin{proof}There is an easy way to see why $g_0$ is Hermitian w.r.t. $J_0$. Let $\mathcal{R}$ be the subbundle of $E|_{\mathcal{O}}$ generated by elements like $X_\varsigma+\xi_\varsigma$ and $Y_\varsigma$. $T\mathcal{O}$ can be identified with $\mathcal{R}$ through the bundle homomorphism $\chi: T\mathcal{O}\rightarrow \mathcal{R}$ defined by
\[\chi(X_\varsigma)=X_\varsigma+\xi_\varsigma,\quad \chi(Y_\varsigma)=Y_\varsigma.\]
 Then $J_0$ is nothing else but the restriction of $-\mathbb{J}_1$ on $\mathcal{R}$ and $g_0$ is just the restriction of the generalized metric $\mathcal{G}$ on $\mathcal{R}$. The compatibility of $J_0$ and $g_0$ is simply that of $\mathbb{J}_1$ and $\mathcal{G}$.

To see $(\mathcal{O}, J_0, g_0)$ is a K$\ddot{a}$hler manifold, we first derive an interesting formula:
\begin{equation}d\mu_\varsigma=-\iota_{X_\varsigma}\omega_0\label{mm}\end{equation}
where $\omega_0=g_0J_0$. This actually means the $G$-action is Hamiltonian if we have proved $\omega_0$ is closed. Eq.~(\ref{mm}) can be checked directly. In fact, by Eq.~(\ref{cent}) and definition,
\[\iota_{X_\zeta}d\mu_\varsigma=-g(Y_\varsigma,X_\zeta)=-g_0(Y_\varsigma,X_\zeta)=-g_0(J_0X_\varsigma,X_\zeta)=-\omega_0(X_\varsigma,X_\zeta),\]
and
\begin{eqnarray*}
\iota_{Y_\zeta}d\mu_\varsigma&=&-g(Y_\varsigma, Y_\zeta)=-g(X_\varsigma^+, X_\zeta^+)=-g_0(X_\varsigma, X_\zeta)=-g_0(J_0X_\varsigma, Y_\zeta)\\
&=&-\omega_0(X_\varsigma, Y_\zeta)
\end{eqnarray*}
as required.

To show $\omega_0$ is closed, we have to prove: (i) $d\omega_0(X_\varsigma, X_\zeta, X_\sigma)=0$; (ii) $d\omega_0(Y_\varsigma, X_\zeta, X_\sigma)=0$; (iii) $d\omega_0(Y_\varsigma, Y_\zeta, X_\sigma)=0$; (iv) $d\omega_0(Y_\varsigma, Y_\zeta, Y_\sigma)=0$. We will only prove the last two and the rest two, which can be checked similarly, will be left to the interested readers.

Proof of (iii).\begin{eqnarray*}
d\omega_0(Y_\varsigma, Y_\zeta, X_\sigma)&=&Y_\varsigma\omega_0(Y_\zeta, X_\sigma)-Y_\zeta\omega_0(Y_\varsigma,X_\sigma)+X_\sigma\omega_0(Y_\varsigma, Y_\zeta)-\omega_0([Y_\varsigma, Y_\zeta], X_\sigma)\\
&+&\omega_0([Y_\varsigma,X_\sigma], Y_\zeta)-\omega_0([Y_\zeta, X_\sigma], Y_\varsigma)\\
&=&Y_\varsigma Y_\zeta \mu_\sigma-Y_\zeta Y_\varsigma\mu_\sigma+ X_\sigma\omega_0(X_\varsigma, X_\zeta)-[Y_\varsigma, Y_\zeta]\mu_\sigma\\
&+&\omega_0([X_\varsigma,X_\sigma], X_\zeta)-\omega_0([X_\zeta, X_\sigma], X_\varsigma)\\
&=&X_\sigma\omega_0(X_\varsigma, X_\zeta)-\omega_0([X_\sigma, X_\varsigma], X_\zeta)-\omega_0(X_\varsigma, [X_\sigma, X_\zeta,])\\
&=&0
\end{eqnarray*}
where we have used Eq.~(\ref{mm}) and the fact that $[X_\varsigma, J_0X_\zeta]=J_0[X_\varsigma, X_\zeta]$.

Proof of (iv).\begin{eqnarray*}
d\omega_0(Y_\varsigma, Y_\zeta, Y_\sigma)&=&Y_\varsigma\omega_0(Y_\zeta, Y_\sigma)-Y_\zeta\omega_0(Y_\varsigma,Y_\sigma)+Y_\sigma\omega_0(Y_\varsigma, Y_\zeta)-\omega_0([Y_\varsigma, Y_\zeta], Y_\sigma)\\
&+&\omega_0([Y_\varsigma,Y_\sigma], Y_\zeta)-\omega_0([Y_\zeta, Y_\sigma], Y_\varsigma)\\
&=&Y_\varsigma\omega_0(X_\zeta, X_\sigma)-Y_\zeta\omega_0(X_\varsigma,X_\sigma)+Y_\sigma\omega_0(X_\varsigma, X_\zeta)+\omega_0([X_\varsigma, X_\zeta], Y_\sigma)\\
&-&\omega_0([X_\varsigma,X_\sigma], Y_\zeta)+\omega_0([X_\zeta, X_\sigma], Y_\varsigma)\\
&=&Y_\varsigma\mu_{[\zeta, \sigma]}-Y_\zeta\mu_{[\varsigma, \sigma]}+Y_\sigma\mu_{[\varsigma, \zeta]}-Y_\sigma\mu_{[\varsigma, \zeta]}+Y_\zeta\mu_{[\varsigma, \sigma]}-Y_\varsigma\mu_{[\zeta, \sigma]}\\
&=&0
\end{eqnarray*}
where we have also used the equality $[J_0X_\varsigma, J_0X_\zeta]=-[X_\varsigma, X_\zeta]$ and the equivariance of $\mu$.
\end{proof}

In \cite{GS2}, the discussion concerning the geometry on those $G^\mathbb{C}$-orbits is aimed at investigating stability conditions in equivariant K$\ddot{a}$hler geometry. The theorem above shows the same argument could apply to our more general setting. We will turn to the details of this "generalized complex stability" in future works.

\end{document}